\documentclass{amsart}
\usepackage{amssymb}
\usepackage{amsfonts}

\newtheorem{theorem}{Theorem}
\theoremstyle{plain}

\newtheorem{lemma}{Lemma}

\numberwithin{equation}{section}
\input{tcilatex}

\begin{document}
\title[self-similar wreath products]{On self-similarity of wreath products
of abelian groups}
\author{Alex C. Dantas}
\address{Universidade Tecnol\'{o}gica Federal do Paran\'{a}, Guarapuava-PR
85053-525, Brasil}
\email{alexcdan@gmail.com}
\author{Said N. Sidki}
\address{Departamento de Matematica, Universidade de Bras\'{\i}lia\\
Bras\'{\i}lia DF 70910-900, Brazil}
\email{ssidki@gmail.com}
\date{January 18, 2017}
\subjclass[2000]{ Primary 20E08, 20F18}
\keywords{Automorphisms of trees, state-closed groups, self-similar groups.\\
We thank the referee for suggestions which improved the original text.}

\begin{abstract}
We prove that in a self-similar wreath product of abelian groups $G=BwrX$,
if $X$ is torsion-free then $B$ is torsion of finite exponent. Therefore, in
particular, the group $\mathbb{Z}wr\mathbb{Z}$ cannot be self-similar
Furthemore, we prove that if $L$ is a self-similar abelian group then $%
L^{\omega }wrC_{2}$ is also self-similar.
\end{abstract}

\maketitle

\section{Introduction}

A group $G$ is self-similar provided for some finite positive integer $m$,
the group has a faithful representation on an infinite regular one-rooted $m$%
-tree~$\mathcal{T}_{m}$ such that the representation is state-closed and is
transitive on the tree's first level. If a group $G$ does not admit such a
representation for any $m$ then we say $G$ is not self-similar. In
determining that a group is not self-similar we will use the language of
virtual endomorphisms of groups. More precisely, a group $G$ is not
self-similar if and only if for any subgroup $H$ of $G$ of finite index and
any homomorphism $f:H\rightarrow G$ there exists a non-trivial subgroup $K$
of $H$ which is normal in $G$ and is $f$-invariant (in the sense $K^{f}\leq
K $).

Which groups admit faithful self-similar representation is an on going topic
of investigation. The first in depth study of this question was undertaken
in \cite{NekSid} and then in book form in \cite{Nek}. Faithful self-similar
representations are known for many individual finitely generated groups
ranging from the torsion groups of Grigorchuk and Gupta-Sidki to free
groups. Such representations have been also studied for the family of
abelian groups \cite{BruSid}, of finitely generated nilpotent groups \cite%
{BerSid}, as well as for arithmetic groups \cite{Kapo}. See \cite{DanSid}
for further references.

One class which has received attention in recent years is that of wreath
products of abelian groups $G=BwrX$, such as the classical lamplighter group 
\cite{GrigZuk} in which $B$ is cyclic of order $2$ and $X$ is infinite
cyclic. The more general class $G=C_{p}wrX$ where $C_{p}$ is cyclic of prime
order $p$ and $X$ is free abelian of rank $d\geq 1$ was the subject of \cite%
{DanSid} where self-similar groups of this type are constructed for every
finite rank $d$.

We show in this paper that the properly of $B$ being torsion is necessary to
guarantee self-similarity of $G$. More precisely, we prove

\begin{theorem}
Let $G=BwrX$ be a self-similar wreath product of abelian groups. If $X$ is
torsion free then $B$ is a torsion group of finite exponent. In particular, $%
\mathbb{Z}wr\mathbb{Z}$ cannot be self-similar.
\end{theorem}

Observe that though $G=\mathbb{Z}wr\mathbb{Z}$ is not self-similar, it has a
faithful finite-state representation on the binary tree \cite{Sidki}.

Next we produce a novel embedding of self-similar abelian groups into
self-similar wreath products having higher cardinality.

\begin{theorem}
Let $L$ be a self-similar abelian group and $L^{\omega }$ an infinite
countable direct sum of copies of $L$. Then $L^{\omega }wrC_{2}$ is also
self-similar.
\end{theorem}

\section{Preliminaries}

We recall a number of notions of groups acting on trees and of virtual
endomorphisms of groups from \cite{BruSid}.

\subsection{State-closed groups}

1. Automorphisms of one-rooted regular trees $\mathcal{T}\left( Y\right) $
indexed by finite sequences from a finite set $Y$ of size $m\geq 2$, have a
natural interpretation as automata on the alphabet $Y$, and with states
which are again automorphisms of the tree. A subgroup $G$ of the group of
automorphisms $\mathcal{A}\left( Y\right) $ of the tree is said to have
degree $m$. Moreover, $G$ is \textit{state-closed }of degree $m$ provided
the states of its elements are themselves elements of the same group.

2. Given an automorphism group $G$ of the tree, $v$ a vertex of the tree and 
$l$ a level of the tree, we let $Fix_{G}(v)$ denote the subgroup of $G$
formed by its elements which fix $v$ and let $Stab_{G}(l)$ denote the
subgroup of $G$ formed by elements which fix all $v$ of level $l$. Also, let 
$P\left( G\right) $ denote the permutation group induced by $G$ on the first
level of the tree. We say $G$ is \textit{transitive} provided $P\left(
G\right) $ is transitive.

3. A group $G$ is said to be \textit{self-similar} provided it is a
state-closed and transitive subgroup of $\mathcal{A}\left( Y\right) $ for
some finite set $Y$.

\subsection{Virtual endomorphisms}

1. Let $G$ be a group with a subgroup $H$ of finite index $m$. A
homomorphism $f:H\rightarrow G$ is called a \textit{virtual endomorphism }of 
$G$ and $\left( G,H,f\right) $ is called a \textit{similarity triple}; if $G$
is fixed then $\left( H,f\right) $ is called a \textit{similarity pair}.

2. Let $G$ be a transitive state-closed subgroup of $\mathcal{A}\left(
Y\right) $ where $Y=\left\{ 1,2,...,m\right\} $. Then the index $\left[
G:Fix_{G}\left( 1\right) \right] =m$ and the projection on the $1$st
coordinate of $Fix_{G}\left( 1\right) $ produces a subgroup of $G$; that is, 
$\pi _{1}:Fix_{G}\left( 1\right) \rightarrow G$ is a virtual endomorphism of 
$G$.

3. Let $G$ be a group with a subgroup $H$ of finite index $m$ and a
homomorphism $f:H\rightarrow G$. If $U\leq H$ and $U^{f}\leq U$ then $U$ is
called $f$-\textit{invariant. }The largest subgroup $K$ of $H$, which is
normal in $G$ and is $f$-invariant is called the $f$-core$\left( H\right) $.
If the $f$-core$\left( H\right) $ is trivial then $f$ and the triple $\left(
G,H,f\right) $ are called \textit{simple.}

4. Given a triple $\left( G,H,f\right) $ and a right transversal $L=\left\{
x_{1},x_{2},...,x_{m}\right\} $ of $H$ in $G$, the permutational
representation $\pi :G\rightarrow Perm\left( 1,2,...,m\right) $ is $g^{\pi
}:i\rightarrow j$ which is induced from the right multiplication $%
Hx_{i}g=Hx_{j}$. Generalizing the Kalujnine-Krasner procedure \cite{Robinson}
, we produce recursively a representation $\varphi :G\rightarrow $ $\mathcal{%
A}\left( Y\right) $, defined by 
\begin{equation*}
g^{\varphi }=\left( \left( x_{i}g.\left( x_{\left( i\right) g^{\pi }}\right)
^{-1}\right) ^{f\varphi }\right) _{1\leq i\leq m}g^{\pi }\text{,}
\end{equation*}%
seen as an element of an infinitely iterated wreath product of $Perm\left(
1,2,...,m\right) $. The kernel of $\varphi $ is precisely the $f$-core$%
\left( H\right) $ and $G^{\varphi }$ is state-closed and transitive and $%
H^{\varphi }=Fix_{G^{\varphi }}(1)$.

\begin{lemma}
A group $G$ is self-similar if and only if there exists a simple similarity
pair $\left( H,f\right) $ for $G$.
\end{lemma}

\section{Proof of Theorem 1}

We recall $B$, $X$ are abelian groups, $X$ is a torsion-free group and $%
G=BwrX$ . Denote the normal closure of $B$ in $G$ by $A=B^{G}$ . Let $\left(
H,f\right) $ be the similarity pair with respect to which $G$ is
self-similar and let $\left[ G:H\right] $ $=m$. Define 
\begin{equation*}
A_{0}=A\cap H,\text{ }L=\left( A_{0}\right) ^{f}\cap A,\text{ }Y=X\cap (AH) 
\text{.}
\end{equation*}
Note that if $x\in X$ is nontrivial then the centralizer $C_{A}\left(
x\right) \ $is trivial. We develop the proof in four lemmas.

\begin{lemma}
\textbf{\ }Either $B^{m}\ $is trivial or $\left( A_{0}\right) ^{f}\leq A$.
In both cases $A\not=A_{0}$.
\end{lemma}

\begin{proof}
\textit{We have }$A^{m}\leq A_{0}$\textit{\ and }$X^{m}\leq H$\textit{. As }$%
A$ is normal abelian and $X$ is abelian, \textit{\ } 
\begin{eqnarray*}
&&\left[ A^{m},X^{m}\right] \vartriangleleft G, \\
\left[ A^{m},X^{m}\right] &\leq &\left[ A_{0},X^{m}\right] \leq A_{0}\text{.}
\end{eqnarray*}%
\textit{Also, } 
\begin{eqnarray*}
f &:&\left[ A^{m},X^{m}\right] \rightarrow \left[ \left( A^{m}\right)
^{f},\left( X^{m}\right) ^{f}\right] \leq \left( A_{0}\right) ^{f}\cap
G^{\prime } \\
&\leq &\left( A_{0}\right) ^{f}\cap A\text{ }=L\text{.}
\end{eqnarray*}%
\textit{(1) If }$L$\textit{\ is trivial then }$\left[ A^{m},X^{m}\right]
\leq \ker \left( f\right) $\textit{. Since }$f$\textit{\ is simple, it
follows that }$\ker \left( f\right) =1$\textit{\ and }$\left[ A^{m},X^{m}%
\right] =1=\left[ B^{m},X^{m}\right] $\textit{. As }$X^{m}$\textit{\ }$%
\not=1 $\textit{, we conclude }$A^{m}=1=B^{m}$\textit{. (2) If }$L$\textit{\
is nontrivial then }$L$\textit{\ is central in }$M=A\left( A_{0}\right)
^{f}=A\left( X\cap M\right) $\textit{\ which implies }$X\cap M=1$\textit{\
and }$\left( A_{0}\right) ^{f}\leq A$\textit{. (3) If }$B$ is a torsion
group then $tor(G)=A$; clearly, $\left( A_{0}\right) ^{f}\leq A$ and $%
A\not=A_{0}$.
\end{proof}

Let $G$ be a counterexample; that is, $B$ has infinite exponent. By the
previous lemma $\left( A_{0}\right) ^{f}\leq A$ and so we may use
Proposition 1 of \cite{DanSid} to replace the simple similarity pair $\left(
H,f\right) $ by a simple pair $\left( \dot{H},\dot{f}\right) $ where $\dot{H}%
=A_{0}Y$ ($Y\leq X$) and $\left( Y\right) ^{\dot{f}}\leq X$. In other words,
we may assume $\left( Y\right) ^{f}\leq X$.

\begin{lemma}
If $z\in X$ is nontrivial and $x_{1},...,x_{t},z_{1},...,z_{l}\in X$, then
there exists an integer $k$ such that 
\begin{equation*}
z^{k}\{z_{1},...,z_{l}\}\cap \{x_{1},...,x_{t}\}=\emptyset .
\end{equation*}
\end{lemma}

\begin{proof}
\textit{Note that the set }$\{k\in \mathbb{Z}\,|\,\{z^{k}z_{j}\}\cap
\{x_{1},...,x_{t}\}\not=\emptyset \}$\textit{\ is finite, for each }$%
j=1,...,l$\textit{. Indeed, otherwise there exist }$k_{1}\not=k_{2}$\textit{%
\ such that }$z^{k_{1}-k_{2}}=1$\textit{, a contradiction.}
\end{proof}

\begin{lemma}
If $x\in X$ is nontrivial, then $\left( x^{m}\right) ^{f}$ is nontrivial.
\end{lemma}

\begin{proof}
\textit{Suppose that there exists a nontrivial }$x\in X$ \textit{such that }$%
x^{m}\in \ker \left( f\right) $\textit{. Then for each }$a\in A$\textit{\
and each }$u\in X$\textit{\ we have} 
\begin{equation*}
(a^{-mu}a^{mux^{m}})^{f}=(a^{-mu})^{f}(a^{mux^{m}})^{f}=(a^{-mu})^{f}\left(
(a^{mu})^{f}\right) ^{(x^{m})^{f}}
\end{equation*}
\begin{equation*}
=(a^{-mu})^{f}(a^{mu})^{f}=1\mathit{,}\text{.}
\end{equation*}
\textit{Since }$A^{m(x^{m}-1)}\leq \ker \left( f\right) $\textit{\ and is
normal in }$G$, \textit{we have a} \textit{contradiction.}
\end{proof}

\begin{lemma}
The subgroup $A^{m}$ is $f$-invariant.
\end{lemma}

\begin{proof}
\textit{Let }$a\in A$\textit{. Consider }$T=\{c_{1},...,c_{r}\}$\textit{\ a
transversal of }$A_{0}$\textit{\ in }$A$\textit{, where }$r$\textit{\ is a
divisor of }$m$\textit{. Since }$A^{m}$\textit{\ is a subgroup of }$A_{0}$ 
\textit{\ and }$A=\oplus _{x\in X}B^{x}$\textit{, there exist }$%
x_{1},...,x_{t}$\textit{\ such that} 
\begin{equation*}
\left\langle \left( c_{i}^{m}\right) ^{f}|\,i=1,...,r\right\rangle \leq
B^{x_{1}}\oplus ...\oplus B^{x_{t}}
\end{equation*}%
\textit{and} $z_{1},...,z_{l}\in X$ such that 
\begin{equation*}
\left\langle (a^{m})^{f}\right\rangle \leq B^{z_{1}}\oplus ...\oplus
B^{z_{l}}\mathit{.}\text{.}
\end{equation*}%
\textit{Since }$[G:H]=m$\textit{, it follows that }$X^{m}\leq Y$. \textit{%
Fix a nontrivial }$x\in X$\textit{\ and let }$z=(x^{m})^{f}$.

\textit{For each integer }$k$, define $i_{k}\in \{1,...,r\}$\textit{\ such
that } 
\begin{equation*}
a^{x^{mk}}c_{i_{k}}^{-1}\in A_{0}\text{.}
\end{equation*}%
\textit{Then} 
\begin{equation*}
\left( (a^{x^{mk}}c_{i_{k}}^{-1})^{m}\right) ^{f}=\left(
(a^{x^{mk}}c_{i_{k}}^{-1})^{f}\right) ^{m}\in A^{m}\mathit{.}
\end{equation*}%
\textit{But }$(a^{x^{mk}}c_{i_{k}}^{-1})^{m}=a^{mx^{mk}}c_{i_{k}}^{-m}$ 
\textit{, thus} 
\begin{equation*}
\left( (a^{x^{mk}}c_{i_{k}}^{-1})^{m}\right) ^{f}=\left( a^{mx^{mk}}\right)
^{f}\left( c_{i_{k}}^{-m}\right) ^{f}=\left( a^{mf}\right)
^{z^{k}}c_{i_{k}}^{-mf}\text{.}
\end{equation*}%
\textit{By Lemma 4, }$z\not=1$\textit{. There exists by Lemma 3 an integer }$%
k^{\prime }$\textit{\ such that}

\begin{equation*}
\{z^{k^{\prime }}z_{1},...,z^{k^{\prime }}z_{l}\}\cap
\{x_{1},...,x_{t}\}=\emptyset \mathit{,}
\end{equation*}
\textit{and so,} 
\begin{equation*}
(B^{z^{k^{\prime }}z_{1}}\oplus ...\oplus B^{z^{k^{\prime }}z_{l}})\cap
(B^{x_{1}}\oplus ...\oplus B^{x_{t}})=1\text{.}
\end{equation*}
\textit{It follows that } 
\begin{equation*}
(a^{mf})^{z^{k^{\prime }}}c_{i_{k}}^{-mf}\in A^{m}\cap \lbrack
(B^{z^{k^{\prime }}z_{1}}\oplus ...\oplus B^{z^{k^{\prime }}z_{l}})\oplus
(B^{x_{1}}\oplus ...\oplus B^{x_{t}})]\text{. }
\end{equation*}
\textit{But as} 
\begin{equation*}
A^{m}=\oplus _{x\in X}B^{mx}\mathit{.}\text{,}
\end{equation*}
\textit{we conclude, }$(a^{mf})^{z^{k^{\prime }}}\in B^{mz^{k^{\prime
}}z_{1}}\oplus ...\oplus B^{mz^{k^{\prime }}z_{l}}\leq A^{m}$\textit{\ and }$%
a^{mf}\in A^{m}$\textit{. Hence, }$\left( A^{m}\right) ^{f}\leq A^{m}$.
\end{proof}

With this last lemma, the proof of Theorem 2 is finished.

\section{\textbf{Proof of Theorem 2 }}

Let $L$ be a self-similar\textit{\ abelian} group with respect to a simple
triple $\left( L,M,\phi \right) $; then $\phi $ is a monomorphism. Define $%
B=\sum_{i\geq 1}L_{i}$ , a direct sum of groups where $L_{i}=L$ for each $i$
. Let $X$ be cyclic group of order $2$ and $G=BwrX$, the wreath product of $%
B $ by $X$. Denote the normal closure of $B$ in $G$ by $A$; then, 
\begin{eqnarray*}
A &=&B^{X}=\left( L_{1}\oplus \sum_{i\geq 2}L_{i}\right) \times B \\
G &=&A\bullet X\text{.}
\end{eqnarray*}
Define the subgroup of $G$ 
\begin{equation*}
H=\left( M\oplus \sum_{i\geq 2}L_{i}\right) \times B\,\,\text{;}
\end{equation*}
an element of $H$ has the form 
\begin{equation*}
\beta =\left( \beta _{1},\beta _{2}\right)
\end{equation*}
where 
\begin{eqnarray*}
\beta _{i} &=&\left( \beta _{ij}\right) _{j\geq 1}\,\text{, }\beta _{ij}\in L
\\
\beta _{1} &=&\left( \beta _{1j}\right) _{j\geq 1},\text{ }\beta _{11}\in M 
\text{.}
\end{eqnarray*}
We note that $[G:H]$ is finite; indeed, 
\begin{equation*}
\lbrack A:H]=[L:M]\text{ and }[G:H]=2[L:M].
\end{equation*}
Define the maps 
\begin{equation*}
\phi _{1}^{\prime }:M\oplus \left( \sum_{i\geq 2}L_{i}\right) \rightarrow B, 
\text{ }\phi _{2}^{\prime }:B\rightarrow B
\end{equation*}
where for $\beta =(\beta _{1},\beta _{2})=\left( (\beta _{1j}),(\beta
_{2j})\right) _{j\geq 1},\beta _{11}\in M,$ 
\begin{equation*}
\phi _{1}^{\prime }:\beta _{1}\mapsto (\beta _{11}^{\phi }\beta _{12},\beta
_{13},...),\phi _{2}^{\prime }:\beta _{2}\mapsto (\beta _{22},\beta
_{21},\beta _{23},...)\text{.}
\end{equation*}
Since $L$ is abelian, $\phi _{1}^{\prime }$ is a homomorphism and clearly $%
\phi _{2}^{\prime }$ is a homomorphism as well.

Define the homomorphism 
\begin{equation*}
f:\left( M\oplus \sum_{i\geq 2}L_{i}\right) \times B\,\rightarrow A
\end{equation*}
by 
\begin{equation*}
f:\left( \beta _{1},\beta _{2}\right) \mapsto \left( \left( \beta
_{1}\right) ^{\phi _{1}^{\prime }},(\beta _{2})^{\phi _{2}^{\prime }}\right) 
\text{.}
\end{equation*}

Suppose by contradiction that $K$ is a nontrivial subgroup of $H$, normal in 
$G$ and $f$-invariant and let $\kappa =\left( \kappa _{1},\kappa _{2}\right) 
$ be a nontrivial element of $K$. Since $X$ permutes transitively the
indices of $\kappa _{i}$, we conclude $\kappa _{i1}\in M$ for $i=1,2$. Let $%
s_{i}$ (call it degree) be the maximum index of the nontrivial entries of $%
\kappa _{i}$; if $\kappa _{i}=0$ then write $s_{i}=0$ . Choose $\kappa $
with minimum $s_{1}+s_{2}$; we may assume $s_{1}$ be minimum among those $%
s_{i}$ $\not=0$. Since 
\begin{eqnarray*}
\kappa _{1} &=&\left( \kappa _{1j}\right) _{j\geq 1},\kappa _{11}\in M, \\
\left( \kappa _{1}\right) ^{\phi _{1}^{\prime }} &=&(\kappa _{11}^{\phi
}\kappa _{12},\kappa _{13},...)\text{,}
\end{eqnarray*}
we conclude $\left( \kappa _{1}\right) ^{\phi _{1}^{\prime }}$ has smaller
degree and therefore 
\begin{equation*}
\kappa _{1}=\left( \kappa _{11},e,e,e,...\right) \text{ or }\left( \kappa
_{11},\kappa _{11}^{-\phi },e,e,...\right) \text{. }
\end{equation*}
Suppose $\kappa _{1}=\left( \kappa _{11},e,e,e,...\right) $. As, $\kappa
=\left( \kappa _{1},\kappa _{2}\right) \in K$, we have $\kappa _{11}\in M$
and therefore 
\begin{eqnarray*}
\kappa ^{f} &=&\left( \left( \kappa _{1}\right) ^{\phi _{1}^{\prime
}},(\kappa _{2})^{\phi _{2}^{\prime }}\right) \in K, \\
\left( \kappa _{1}\right) ^{\phi ^{\prime }} &=&\left( \left( \kappa
_{11}\right) ^{\phi },e,e,e,...\right) , \\
\left( \kappa _{11}\right) ^{\phi } &\in &M;
\end{eqnarray*}
\begin{eqnarray*}
\kappa ^{f^{2}} &=&\left( \left( \kappa _{1}\right) ^{\left( \phi
_{1}^{\prime }\right) ^{2}},\kappa _{2}\right) , \\
\left( \kappa _{1}\right) ^{\left( \phi _{1}^{\prime }\right) ^{2}}
&=&\left( \left( \kappa _{11}\right) ^{\phi ^{2}},e,e,e,...\right) , \\
\left( \kappa _{11}\right) ^{\phi ^{2}} &\in &M\text{; etc.}
\end{eqnarray*}
By simplicity of $\phi $, this alternative is out. That is, 
\begin{equation*}
\kappa _{1}=\left( \kappa _{11},\kappa _{11}^{-\phi },e,e,...\right) \text{, 
} \kappa _{11}\in M\text{ .}
\end{equation*}
Therefore 
\begin{eqnarray*}
\kappa &=&\left( \kappa _{1},\kappa _{2}\right) , \\
\kappa ^{x} &=&\left( \kappa _{2},\kappa _{1}\right) ,\text{ }\kappa
^{xf}=\left( \left( \kappa _{2}\right) ^{\phi _{1}^{\prime }},\left( \kappa
_{11}^{-\phi },\kappa _{11},e,e,...\right) \right) , \\
\kappa ^{xfx} &=&\left( \left( \kappa _{11}^{-\phi },\kappa
_{11},e,e,...\right) ,\left( \kappa _{2}\right) ^{\phi _{1}^{\prime }}\right)
\end{eqnarray*}
are elements of $K$ and so, $\kappa _{11}^{\phi }\in M$. Furthermore,

\begin{eqnarray*}
\kappa ^{xfxf} &=&\left( \left( \kappa _{11}^{-\phi ^{2}}\kappa
_{11},e,e,...\right) ,\left( \kappa _{2}\right) ^{\phi _{1}^{\prime }\phi
_{2}^{\prime }}\right) , \\
\kappa _{11}^{-\phi ^{2}}\kappa _{11} &\in &M\text{;}
\end{eqnarray*}
successive applications of $f$ to $\kappa ^{xfx}$ produces $\kappa
_{11}^{\phi ^{i}}\in M$. Therefore, $\left\langle \kappa _{11}^{\phi
^{i}}\mid i\geq 0\right\rangle $ is a $\phi $-invariant subgroup of $M$; a
contradiction is reached.

\end{document}